\documentclass[12pt]{article}
\usepackage[a4paper]{geometry}

\usepackage{amsmath,amsthm,amssymb}
\usepackage{color,graphicx}
\usepackage{algorithm,algorithmic,caption}

\definecolor{gold}{rgb}{.7,.5,0}
\definecolor{dred}{rgb}{0.92,0,0}
\definecolor{dgreen}{rgb}{0,0.6,0}

\newcommand\GG{\mathbb G}

\newcommand\cB{\mathcal B}
\newcommand\cH{\mathcal H}
\newcommand\cG{\mathcal G}
\newcommand\cM{\mathcal M}
\newcommand\cN{\mathcal N}
\newcommand\cQ{\mathcal Q}
\newcommand\cT{\mathcal T}

\newcommand{\RR}{\mathbb R}

\DeclareMathOperator\diam{diam}

\DeclareMathOperator\supp{supp}
\DeclareMathOperator\trunc{trunc}
\DeclareMathOperator\Trunc{Trunc}

\DeclareMathOperator \refine {REFINE}
\DeclareMathOperator \refrec {REFINE\_RECURSIVE}

\DeclareMathOperator \Dist {dist}
\newcounter{num}

\newcommand\coloneqq{:=}
\let\hat\widehat

\newcommand \mathbox [2][2.5em]{\makebox[#1]{$\displaystyle #2$}}

\newtheorem{thm}{Theorem}
\newtheorem{dfn}[thm]{Definition}

\newtheorem{lma}[thm]{Lemma}
\newtheorem{crl}[thm]{Corollary}
\newtheorem{prn}[thm]{Proposition}

\def\M3AS{Math.\ Models\ Methods\ Appl.\ Sci.}

\begin{document}

\title{\large \bfseries COMPLEXITY OF HIERARCHICAL REFINEMENT \\ FOR A CLASS OF ADMISSIBLE MESH CONFIGURATIONS}
%Complexity of hierarchical refinement for a class of admissible mesh configurations 

\author{
\normalsize ANNALISA BUFFA\footnote{Istituto di Matematica Applicata e Tecnologie Informatiche `E.~Magenes' del CNR,  via Ferrata 1, 27100 Pavia Italy. E-mail address: annalisa@imati.cnr.it}\,, 
CARLOTTA GIANNELLI\footnote{Istituto Nazionale di Alta Matematica, 
Unit\`a di Ricerca di Firenze c/o DiMaI `U.~Dini', Universit\`a di Firenze, viale Morgagni 67a, 50134 Firenze, Italy. E-mail address:  carlotta.giannelli@unifi.it}\,,\\[-.2em]
\normalsize PHILIPP MORGENSTERN\footnote{Institute for Numerical Simulation, Rheinische Friedrich-Wilhelms-Universit{\"a}t Bonn,\break Wegelerstr. 6, 53115 Bonn, Germany. E-mail addresses: morgenstern@ins.uni.bonn.de, peterseim@ins.uni.bonn.de}, DANIEL PETERSEIM\footnotemark[\value{footnote}]
} 

\date{}

\maketitle

\begin{abstract}\noindent
 An adaptive isogeometric method based on $d$-variate hierarchical spline constructions can be derived by considering a refine module that preserves a certain class of admissibility between two consecutive steps of the adaptive loop \cite{buffa2015a}. 

In this paper we provide a complexity estimate, i.e., an estimate on how the number of { mesh elements grows with respect to the number of elements} that are marked for refinement by the adaptive strategy. Our estimate is in the line of the similar ones proved in the finite element context,  \cite{binev2004a, stevenson2007}. 
\end{abstract}

\centerline{{\bf Keywords}: {isogeometric analysis; hierarchical splines; THB-splines; adaptivity.}}
\setcounter{tocdepth}{1}
%\tableofcontents

%--------------------------------------------------------------------------------
\section{Introduction}
\label{sec:intro}
 
Throughout the last years,  Isogeometric Methods have gained widespread interest and are a very active field of research,  \cite{cottrell2009,daveiga2014},   investigating a wide range of applications and theoretical questions. 
Due to the tensor-product structure of splines, there exists very stable procedure to perform mesh refinement and degree raising which are known in the literature as $h$-refinement, $p$-refinement, $k$-refinement \cite{cottrell2009}. While these algorithms are very efficient, the preservation of the tensor-product structure at least locally to each patch, produces a dramatic increase of degrees of freedom together with elongated elements. Mainly for this reason, several approaches have been proposed to alleviate these constraints and they all need the definition of B-splines over non-tensor-product meshes. Indeed, there are several strategies and we mention here T-splines \cite{bazilevs2010},  hierarchical B-splines \cite{forsey1988,kraft1997,kvvb2014} and THB-splines \cite{giannelli2012}, but also LR splines \cite{dokken2013,bressan2013}, hierarchical T-splines \cite{evans2015}, modified T-splines \cite{kang2013}, PHT-splines \cite{deng2008,wang2011c} amongst others. 

Clearly, the development of adaptive strategies exploiting the potential of non-tensor-product splines 
is an interesting and important step which has been approached in a number of papers, at least from the practical point of view. In fact, despite their performance in experiments \cite{bazilevs2010,doerfel2010,daveiga2014,kvvb2014,evans2015}, the advantages of mesh-adaptive Isogeometric Methods have not been assessed in theory until today. % We aim to generalize existing results from (mesh-)Adaptive Finite Element Methods \cite{binev2004a,stevenson2007,cascon2008,cfpp2014} that prove optimality of the convergence rates.
Partial results, in particular on approximation, efficient and reliable error estimates, and convergence of the adaptive procedure, have been proven in preliminary work \cite{buffa2015a} in the context of (truncated) hierarchical splines,  and we aim at continuing this study providing further  ingredients that are needed to prove the optimal convergence of the proposed adaptive approach in the spirit of Adaptive Finite Element Methods \cite{binev2004a,stevenson2007,cascon2008,cfpp2014}. 

In particular, in this paper we address the complexity of the mesh refinement procedure proposed in \cite{buffa2015a}.  The relation between the set of marked elements and the overall number of refined elements introduced by the refine module is not straightforward: additional elements may be refined to create only (strictly) \emph{admissible} meshes. 
%Admissibility is a constraint that guarantees to develop the adaptivity analysis of hierarchical isogeomet ric methods in the context of suitable graded meshes.
Admissibility is a restriction to suitably graded meshes that allows for the adaptivity analysis of hierarchical isogeometric methods. By starting from an initial mesh configuration, the complexity estimate provides a bound for the ratio of the newly inserted elements and the cumulative number of elements marked for refinement in the subdivision process that leads from the initial to the final mesh. This allows to control the propagation of the refinement beyond the set of elements initially selected by the marking strategy. 
An analogous  complexity analysis is currently available for bivariate and trivariate T-splines \cite{mp2015,morgenstern2015}.

This paper is organized as follows. In Section~\ref{sec:thb}, we recall notation and basic results from \cite{buffa2015a}. Section~\ref{sec:complexity} is devoted to the announced complexity estimate. Conclusions and an outlook to future work are given in Section~\ref{sec:conclusions}.

%--------------------------------------------------------------------------------

%--------------------------------------------------------------------------------
\section{Hierarchical refinement}
\label{sec:thb}
%--------------------------------------------------------------------------------
In this section, we recall notation and basic results from  \cite{buffa2015a}.
Since the complexity analysis of the $\refine$ module can be directly performed in the parametric setting, we avoid to introduce the two different notations for parametric/physical domains (corresponding to with/without the hat symbol $\hat{\cdot}$ in \cite{buffa2015a}).
% \todo{Assume unit hypercubes.}

\subsection{The truncated hierarchical basis}

Let $V^0\subset V^1\subset\dots\subset V^{N-1}$ be a nested sequence of  tensor-product $d$-variate spline spaces defined on a closed hypercube $D$ in $\RR^d$. For each level $\ell$, with $\ell=0,1,\dots,N-1$, we denote by $\cB^\ell$ the normalized tensor-product B-spline basis of the spline space $V^\ell$ defined on the grid $G^\ell$, and we assume that $G^0$ consists of open hypercubes with side length 1. 
The Cartesian product of $d$ open intervals between adjacent (and non-coincident) grid values defines a quadrilateral element $Q$ of $G^\ell$. 
For all $Q\in G^k$ we denote by  $h_Q \coloneqq 2^{-k}$  the length of its side,  and by $\ell(Q)$ its level, i.e., $\ell(Q) = k$. Moreover, we assume a fixed degree $\mathbf{p}=(p_1,\ldots,p_d)$ at any hierarchical level. 
The analysis could be generalized to the more general case of a non-uniform initial knot configuration  by suitably taking into account the corresponding maximum local mesh size and to variable degrees as well. 

In order to define hierarchical spline spaces, we additionally consider a nested sequence of domains $\Omega^0\supseteq\Omega^1\supseteq\ldots\supseteq\Omega^{N-1}$, that are  closed  subsets of $D$.  Any $\Omega^\ell$ is the union of the closure of elements that belong to the tensor-product grid of the previous level.
The \emph{hierarchical mesh} $\cQ$ is defined as 
\begin{equation}\label{eq:mesh}
\cQ \,:=\, \left\{ Q\in \cG^\ell, \,\ell=0,\dots,N-1 \right\},
\end{equation}
where 
\begin{equation}\label{eq:active}
\cG^\ell := \left\{Q\in G^\ell : 
Q\subset\Omega^\ell \wedge
\nexists\; Q^*\in G^{\ell^*},\, \ell^*>\ell :\, Q^*\subset\Omega^{\ell^*} \,\wedge\, Q^* \subset Q\right\}
\end{equation}
is the set of active elements of level $\ell$.
We say that $\cQ^*$ is a refinement of $\cQ$, and denote $\cQ^*\succeq\cQ$, when $\cQ^*$ is obtained from $\cQ$ by splitting some of its elements via ``$q$-adic'' refinement. Although the hierarchical approach allows us to consider any integer $q\ge2$, we will focus on the case of standard dyadic refinement with $q=2$.
Hierarchical B-splines are constructed according to a selection of active basis functions at different levels of detail, see also \cite{kraft1997,vuong2011}.
\begin{dfn}\label{dfn:hb}
The hierarchical B-spline (HB-spline) basis $\cH$ with respect to the mesh $\cQ$ is defined as
\begin{equation*}
\cH(\cQ) := \left\{
\beta\in\cB^\ell : 
\supp \beta \subseteq\Omega^\ell \wedge 
\supp \beta\not\subseteq \Omega^{\ell+1}, \, \ell=0,\ldots,N-1
\right\},
\end{equation*}
where $\supp \beta$ denotes the intersection of the support of $\beta$ with $\Omega^0$.
\end{dfn}

The following definition introduces the \emph{truncation} mechanism, the key concept used to define the truncated basis for hierarchical splines \cite{giannelli2012}.
\begin{dfn}\label{dfn:trunc}
 Let 
\begin{equation*}\label{eq:2scale}
s=\sum_{\beta\in \cB^{\ell+1}} c_\beta^{\ell+1}(s) \beta,
\end{equation*}
be the representation of $s\in V^\ell\subset V^{\ell+1}$ with respect to the finer basis $\cB^{\ell+1}$. The truncation of $s$ with respect to $\cB^{\ell+1}$ is defined as
\begin{equation*}\label{eq:trunc}
\trunc^{\ell+1} s \coloneqq \mathbox{\sum_{\substack{\beta\in {\cB}^{\ell+1} \\ \supp\beta\nsubseteq\Omega^{\ell+1}}}} c_\beta^{\ell+1}(s) \beta.
\end{equation*}
\end{dfn}

%The truncated hierarchical basis is defined as follows \cite{giannelli2012}.

\begin{dfn}\label{dfn:thb}
The truncated hierarchical B-spline (THB-spline) basis $\cT$ with respect to the mesh $\cQ$ is defined as
\begin{equation*}
\cT(\cQ) \,:=\, 
\left\{ \Trunc^{\ell+1}\,\beta:\beta\in\cB^\ell
\cap\cH(\cQ),\, \ell=0,\ldots,N-1\right\}, 
\end{equation*}
where $\Trunc^{\ell+1}\beta \coloneqq \trunc^{N-1}(\trunc^{N-2}(\ldots (\trunc^{\ell+1}(\beta))\dots))$, for any $\beta\in\cB^\ell\cap\cH(\cQ)$.
\end{dfn}

For details on the properties of the truncated basis, we refer to \cite{giannelli2012,giannelli2014}.

\subsection{Admissible meshes}

We consider the class of \emph{admissible meshes} introduced in \cite{buffa2015a}.
\begin{dfn}\label{dfn:amesh} 
A mesh $\cQ$ is admissible of class $m$ if the truncated basis functions in $\cT(\cQ)$ which take non-zero values over any element $Q\in\cQ$ belong to at most $m$ successive levels.
\end{dfn}

Since the case $m=1$ refers to uniform meshes, we will from now on focus on the case $m\ge2$.
The relevance of admissible mesh configurations relies on two properties  of  THB-splines.  First, for each element $Q$ of an admissible mesh,  the number of truncated basis functions of degree $\mathbf{p}=(p_1, \ldots, p_d)$ which are non-zero on $Q$ is { less than} $m\prod_{i=1}^d(p_i+1)$.  
Second, if $\cQ$ is an admissible mesh of class $m$, 
then for all truncated basis functions $\tau\in \cT(\cQ)$ and elements $Q\in\cQ$ with $Q\cap\supp \tau \neq \emptyset$,  we have
%\begin{equation}
%  \label{eq:supp-basis}
$ |Q| \lesssim  | \supp \tau | \lesssim |Q|$,
%\end{equation}
where the hidden constants in these inequalities depend on $m$ but not on $\tau$, $\cQ$ and $N$. These properties may be suitably exploited in the analysis of adaptive isogeometric methods. 

In order to characterize a certain class of admissible meshes, we consider the generalization of the \emph{support extension} usually considered in the tensor-product B-spline case to hierarchical configurations.
\begin{dfn}\label{dfn:hse}
The support extension $S(Q,k) $ of an element $Q\in G^\ell$ with respect to level $k$, with $0\le k\le \ell$,
is defined as
\[ S(Q,k) \coloneqq  \left\{ Q'\in G^k: 
\exists\,\beta\in\cB^k,\ \supp\beta\cap Q'\ne\emptyset \wedge \supp\beta\cap Q\ne\emptyset \right\}. \]
\end{dfn}
 By a slight abuse of notation, we will also denote by $S({Q},k) $ the region occupied by the closure of elements in $S({Q},k)$.  A relevant subset of admissible meshes can be defined according to the result of Proposition 9 in \cite{buffa2015a}.
\begin{dfn}\label{dfn:samesh} 
Let $Q$ be the mesh of active elements defined according to \eqref{eq:mesh} and \eqref{eq:active} with respect to the domain hierarchy $\Omega^0\supseteq\Omega^1\supseteq\ldots\supseteq\Omega^{N-1}$. 
A mesh $\cQ$ is strictly admissible of class $m$ if
\begin{equation}\label{eq:sameshes}
\Omega^\ell\subseteq {\omega}^{\ell-m+1} 
\end{equation}
where
\[ 
{\omega}^{\ell-m+1} \coloneqq 
\bigcup \left\{
\overline{Q}: Q\in G^{\ell-m+1} \wedge 
S(Q,{\ell-m+1})\subseteq {\Omega}^{\ell-m+1}\right\},
\]
for $\ell=m,m+1,\ldots,N-1$.
\end{dfn}

The \emph{overlay} $\cQ_*$ of two meshes $\cQ_1, \cQ_2$ is  the mesh obtained as the coarsest common refinement of $\cQ_1$ and $\cQ_2$, usually indicated as 
\[
\cQ_* = \cQ_1 \otimes \cQ_2.
\]
Let $\{\Omega_1^\ell\}_{\ell=0,\ldots,N_1-1}$ and $\{\Omega_2^\ell\}_{\ell=0,\ldots,N_2-1}$ be the nested sequence of domains that define the hierarchical meshes $\cQ_1$ and $\cQ_2$, respectively, with $\Omega_1^0 = \Omega_2^0$. The domain hierarchy $\{\Omega_*^\ell\}_{\ell=0,\ldots,N_*-1}$, with $N_* = \max(N_1,N_2)$, associated to $\cQ_*$ satisfies
\[
\Omega^\ell_* = \Omega_1^\ell \cup \Omega_2^\ell
\quad\text{and}\quad
\omega^\ell_* \supseteq \omega_1^\ell \cup \omega_2^\ell 
\]
for $\ell=1,\ldots,N_*-1$, where $\Omega_i^\ell=\emptyset$ when $\ell\ge N_i$, for i=1,2. Consequently, if $\cQ_1$ and $\cQ_2$ are strictly admissible, for any level $\ell$, we have
\[
\Omega^{\ell}_* \subseteq w_*^{\ell-m+1}
\]
since any $Q\in\Omega^{\ell}_*$ either belongs to $\Omega^{\ell}_1$ or $\Omega^{\ell}_2$ and the overlay $\cQ_*$ is a refinement of both $\cQ_1$ and $\cQ_2$. The overlay $\cQ_*$ of two strictly admissible meshes is then a strictly admissible mesh.
Note that the number of elements of the overlay mesh $\cQ_*$ is bounded as follows, see e.g., \cite{bonito2010,mp2015},
\[
\# \cQ_* = \# (\cQ_1\otimes \cQ_2) 
\le \#\cQ_1 + \#\cQ_2 - \cQ_0,
\]
where $\cQ_0$ is the initial mesh configuration. The above inequality may be used for discussing the optimality of the resulting adaptive isogeometric method, analogously to  adaptive finite element setting.

\subsection{The $\refine$ module}

In order to define an automatic strategy to steer the adaptive method, we will propagate the refinement procedure over a \emph{suitable neighborhood} of any marked element.

\begin{dfn}\label{dfn:neigh}
The neighborhood of $Q\in  \cQ \,\cap\, \cG^\ell$ with respect to $m$ is defined as
\[
\cN(\cQ,Q,m) \coloneqq \left\{Q'\in\cG^{\ell-m+1}: \exists\, Q''  \in S(Q,\ell-m+2), Q''\subseteq Q'\right\},
\]
when { $\ell-m+1 \ge 0$, and $\cN(\cQ, Q,m) = \emptyset$ for $\ell-m+1 < 0$.}
\end{dfn}

A sequence of \emph{strictly} admissible meshes can be recursively defined by suitably extending the refinement of coarser regions beyond the set of marked elements $\cM$ through the algorithms presented in Figure~\ref{fig:refine}. 

Note that these algorithms follow the structure of informal high-level descriptions in the spirit of the analogous modules related to the adaptive finite element methods.

\begin{figure}[ht]
\begin{minipage}{.49\textwidth}
\begin{algorithm}[H]
\footnotesize
\caption*{\footnotesize $\cQ^* = \refine(\cQ,\cM,m)$}
\begin{algorithmic}
\FORALL {$Q\in\cQ\cap\cM$}
\STATE {$\cQ = \refrec(\cQ,Q,m)$}
\ENDFOR
\STATE {$\cQ^*=\cQ$}
\end{algorithmic}
\end{algorithm}
\end{minipage}\hspace{.02\textwidth}%
\begin{minipage}{.49\textwidth}
\begin{algorithm}[H]
\footnotesize
\caption*{\footnotesize $\cQ = \refrec(\cQ,Q,m)$}
\begin{algorithmic}
\FORALL {$Q' \in\cN( \cQ, Q,m)$}
\STATE {$\cQ = \refrec(\cQ,Q',m)$}
\ENDFOR
%\IF {$Q$ has not been subdivided yet}
\STATE {subdivide $Q$ and}
\STATE {update  $\cQ$  by replacing $Q$ with its children}
%\ENDIF
\end{algorithmic}
\end{algorithm}
\end{minipage}
\caption{The $\refine$ and $\refrec$ modules.}\label{fig:refine}
\end{figure}

 By exploiting key properties of the $\refrec$ module,  summarized in Lemma~\ref{lma:rr} and Proposition~\ref{prn:rr} below, Corollary~\ref{crl:refine} characterizes the output of the $\refine$ procedure \cite{buffa2015a}.

 \begin{lma}[Recursive refinement]\label{lma:rr}%
 Let $\cQ$ be a {strictly} admissible mesh of class $m$ and $Q\in \cQ$. The call to $\cQ^* =\refrec(\cQ,Q,m)$ terminates and returns a refined mesh $\cQ^*$ with elements that either were already active in $\cQ$ or are obtained by single refinement of an element of $\cQ$.
 \end{lma}
 
 In addition, if $Q\in \cG^\ell$, the level $\ell^*$ of all newly created elements $Q^*$ generated by the call to $\cQ^* =\refrec(\cQ,Q,m)$ satisfies
\begin{equation}\label{eq:rr}
\ell^*\le \ell+1\,.
\end{equation}
In order to verify this, we may note that the recursion is applied to elements of level $< \ell$, and, in particular, of level $\le \ell-m+1$. If $Q^*$ is a child of $Q$ then $\ell^*=\ell+1$. Otherwise, $Q^*$ is obtained by splitting some elements in the chain of neighborhoods generated by set of recursive calls and, consequently, $\ell^*\le \ell-m+2< \ell+1$ since $m\ge 2$.

 \begin{prn}\label{prn:rr} 
 Let $\cQ$ be a strictly admissible mesh of class $m\ge 2$ and let $Q\in \cG^\ell$,  for some $0\le\ell\le N-1$. Then it follows that the call to $\cQ^* = \refrec(\cQ,Q,m)$ returns a strictly admissible mesh $\cQ^*\succeq\cQ$ of class $m$. 
 \end{prn}
 
 \begin{crl}\label{crl:refine}
Let $\cQ$ be a strictly admissible mesh of class $m\ge 2$ and $\cM$ the set of elements of $\cQ$ marked for refinement. The call to $\cQ^* =$ REFINE $(\cQ,\cM,m)$ terminates and returns a  strictly admissible mesh $\cQ^*\succeq\cQ$ of class $m$. 
\end{crl}

%\begin{prn}%\label{prn:refine}
%The $\refine$ module preserves the class of admissibility of the initial mesh $\cQ$, i.e., for any strictly admissible mesh $\cQ$ of class $m\ge 2$ and $\cM$ the set of elements of $\cQ$ marked for refinement, the call to $\cQ^* =\refine(\cQ,\cM,m)$ terminates and returns a  strictly admissible mesh $\cQ^*\succeq\cQ$ of class $m$. 
%\end{prn}

% in order to keep the refinement procedure as simple as possible
%Note that in each call to $\refrec$, the neighborhood $\cN(\cQ,Q,m)$ considers the support extension of the  coarsest  level that is acceptable for preserving the class of admissibility of the mesh when subdividing the given element $Q$.

Note that in each computation of the neighborhood $\cN(\cQ,Q,m)$, the
choice of level $\ell-m+2$ for the support extension yields the smallest
neighborhood that is acceptable for preserving the class of admissibility of the mesh when subdividing the given element $Q$. Nevertheless, depending on the underlying hierarchical mesh configurations, the basis functions could be also truncated at different intermediate levels.

%----------------------------------------------------------------------
\section{Linear Complexity}\label{sec:complexity}
%----------------------------------------------------------------------

This section is devoted to a complexity estimate in the style proposed by Binev, Dahmen and DeVore \cite{binev2004a} and, in an alternative version, by Stevenson~\cite{stevenson2007}, for adaptive finite element methods.

\subsection{Auxilliary results}

%By recalling that we assume $G^0$ consisting only of translated unit hypercubes and that the level of any  $Q\in G^k$ is $\ell(Q)=k$, the local mesh size $h_Q=|Q|^{1/d}$ that corresponds to the side of $Q$ is given by $2^{-k}$, while its diagonal is $h_Q\sqrt{d}= 2^{-k}\sqrt{d}$. 

For every pair of mesh elements $(Q,Q')$, let $\Dist(Q,Q')$ be the Euclidean distance of their midpoints. Given a $Q\in \cG^\ell$,  all $Q'\in{\cN}(\cQ,Q,m)$ %of level $\ell-m+1$ 
satisfy  
\[
\Dist(Q,Q') \le 
\frac{\sqrt{d}}{2}\,\diam (S(Q,\ell-m+2))\,,
\]
where $\ell=\ell(Q)$ and
\[
\diam (S(Q,\ell-m+2)) := 2^{-\ell+m-2}\, (2\,p+1)
= 2^{-\ell} C_s\,,
\]
with $C_s = C_s(p,m) :=2^{m-2}(2\,p+1)$, $p:=\max_{i=1,\ldots,d} p_i$. Hence,
\begin{equation}\label{eq:qq'}
\Dist(Q,Q') \le 2^{-\ell-1}\, C_d,
\qquad
C_d = C_d(d,p,m) := \sqrt{d}\, C_s\,.
\end{equation}

\begin{lma}\label{lma:dist}
Let $\cQ$ be a strictly admissible mesh of class $m\ge 2$, $\cM$ the set of elements of $\cQ$ marked for refinement, and $Q'\in\cQ\cap\cM$. Any newly created $Q\in\cQ^*\setminus\cQ$ obtained by the call to $\cQ^* = \refrec(\cQ,Q',m)$ satisfies
\begin{equation}\label{eq:c}
\Dist(Q,Q') \le 2^{-\ell(Q)} C
\quad \text{with}\quad 
C \,:=\, \sqrt{d}\,\tilde C,\quad\tilde C\coloneqq \left(2^{-1} + \frac{2}{1-2^{1-m}} C_s\right)\,,
\end{equation}
where then $C$ depends on $d, p$ and $m$.  
\end{lma}
\begin{proof}
The existence of $Q\in\cQ^*\setminus\cQ $ means that $\refrec$ is called over a sequence of elements
$Q' = Q_J, Q_{J-1},\ldots,Q_0$ and corresponding meshes $\cQ_J,\ldots,\cQ_0$ so that $Q_{j-1}\in \cN(\cQ_j,Q_j,m)$, with $Q'\in\cM$ and $Q$ being a child of $Q_0$, namely $\ell(Q)=\ell(Q_0)+1$. Since
$\ell(Q_{j-1})=\ell(Q_j)-m+1$, it %for $j=J,\dots,1$,
follows 
\begin{equation}\label{eq:lj}
\ell(Q_j)=\ell(Q_0)+j\,(m-1).
\end{equation}
We have
\[
\Dist(Q,Q')  \,\le\, \Dist(Q,Q_0)  \,+\,  \Dist(Q_0,Q')
\]
and
\[
\Dist(Q,Q_0) \,=\, 
2^{-\ell(Q)} 2^{-1} \sqrt{d}\,,
\qquad
\Dist(Q_0,Q') \,\le\, \sum_{j=1}^{J} \Dist(Q_j,Q_{j-1})\,.
\]
According to \eqref{eq:qq'} and \eqref{eq:lj}, we obtain
\begin{align*}
\sum_{j=1}^{J} \Dist(Q_j,Q_{j-1}) &\,\le\,
\sum_{j=1}^{J} 2^{-\ell(Q_j)-1}\, C_d \,=\,
\sum_{j=1}^{J} 2^{-\ell(Q_0)-1-j(m-1)}\, C_d \\
&\,<\,
2^{-\ell(Q_0)} C_d \sum_{j=0}^{\infty} 2^{-j(m-1)}\, 
\,=\,
\frac{2^{-\ell(Q_0)}}{1-2^{1-m}} C_d
\,=\,
\frac{2^{-\ell(Q)+1}}{1-2^{1-m}} C_d\,.
\end{align*}
Hence,
$\Dist(Q,Q') \,\le\,
2^{-\ell(Q)} C
$, where $C$ is the constant defined in \eqref{eq:c}.
\end{proof}

\subsection{Main result}
The main result of this paper states the existence of a generic constant $\Lambda = \Lambda(d,p,m)<\infty$  that bounds, for any sequence of successive refinements,  the ratio  between  the number of new elements in the final mesh $\cQ_J$ and the number of all marked elements  encountered in the refinement process from $\cQ_0$ to $\cQ_J$. 

\begin{thm}[Complexity of $\refine$]\label{thm: complexity} 
 Let $\cM\coloneqq \bigcup_{j=0}^{J-1} \cM_j$ be the set of marked elements used to generate the sequence of strictly admissible meshes  $\cQ_0,\cQ_1,\dots,\cQ_J$ starting from $\cQ_0=G^0$, namely 
\[\cQ_j=\refine(\cQ_{j-1},\cM_{j-1},m),\quad\cM_{j-1}\subseteq\cQ_{j-1}\quad\text{for}\enspace j\in\{1,\dots,J\}\,.
\] 
 Then, there exists a constant $\Lambda>0$ so that 
\[
\#\cQ_J - \#\cQ_0 
\le\ \Lambda \sum_{j=0}^{J-1} \# \cM_j\,,
\]
with $\Lambda = \Lambda(d,p,m)\coloneqq 4(4\tilde{C}+1)^d$, where $\tilde{C}$ is defined in \eqref{eq:c}.

\end{thm}
\begin{proof}
We denote by $\GG\coloneqq \bigcup_j G^j$ the set of the initial mesh elements and all elements that can be generated from their successive dyadic subdivision.  Let $Q\in\GG$, $Q_\cM\in\cM$, let  
\[
\lambda(Q,Q_\cM)\coloneqq
\begin{cases}
2^{\ell(Q)-\ell(Q_\cM)}&\text{if }\ell(Q)\le\ell(Q_\cM)+1\text{ and }\Dist(Q,Q_\cM)  <  2^{1-\ell(Q)}\, C,  
\\[.3em]
0&\text{otherwise.}
\end{cases}
\]
 The proof consists of two main steps devoted to identify
\begin{itemize}
\item[(i)] a lower bound for the sum of the $\lambda$ function as $Q_{\cM}$ varies in $\cM$ so that  each $Q\in\cQ_J\setminus\cQ_0$ satisfies 
\begin{equation}\label{eq:lb}
\sum_{Q_\cM\in\cM}\lambda(Q,Q_\cM)\ \ge\ 1\,;
\end{equation}

\item[(ii)] an upper bound for the sum of the $\lambda$ function as the refined element $Q$ varies in $\cQ_J\setminus\cQ_0$ so that, for any $j = 0,\dots,J-1$, each $Q_\cM\in\cM_j$ satisfies 
\begin{equation}\label{eq:ub}
\sum_{Q\in\cQ_J\setminus\cQ_0}\lambda(Q,Q_\cM)
\ \le \Lambda\,.
\end{equation}
\end{itemize}
If inequalities \eqref{eq:lb} and \eqref{eq:ub} hold for a certain constant $\Lambda$, we have 
\begin{align*}
 \#\cQ_J - \#\cQ_0  &= \sum_{Q\in\cQ_J\setminus\cQ_0} 1
\,\le\,\sum_{Q\in\cQ_J\setminus\cQ_0}\sum_{Q_\cM\in\cM}\lambda(Q,Q_\cM) \\
&
\le\sum_{Q_\cM\in\cM}  \Lambda 
=  \Lambda  \,\sum_{j=0}^{J-1}  \# \cM_j\,, 
\end{align*}
 and the proof of the theorem is complete. We detail below the analysis of (i) and (ii).

(i)\enspace Let $Q\in\cQ_J\setminus\cQ_0$ be an element generated in the refinement process from $\cQ_0$ to $\cQ_J$, and let $j_1<J$ be the index so that  $Q\in\cQ_{j_1+1}\setminus\cQ_{j_1}$.  Lemma~\ref{lma:dist} together with \eqref{eq:rr} state the existence of $Q_1\in\cM_{j_1}$ with 
\[
\Dist(Q,Q_1)\le 2^{-\ell(Q)}\, C 
\quad \text{and}\quad 
\ell(Q)\le\ell(Q_1)+1\,, 
\]  
and, consequently $\lambda(Q,Q_1)=2^{\ell(Q)-\ell(Q_1)}>0$.
The repeated use of Lemma~\ref{lma:dist} yields  a sequence $\{Q_2,Q_3,\dots\}$  with $Q_{i-1}\in\cQ_{j_i+1}\setminus\cQ_{j_i}$,  for $j_1>j_2>j_3>\dots$, and $Q_i\in\cM_{j_i}$ such that 
\begin{equation}\label{eq: complexity -last}
\Dist(Q_{i-1},Q_i)\le 2^{-\ell(Q_{i-1})}\, C 
\quad\text{and}\quad 
\ell(Q_{i-1})\le\ell(Q_i)+1.
\end{equation}
We iteratively apply Lemma~\ref{lma:dist} as long as 
\[
\lambda(Q,Q_i)>0 \quad \text{and} \quad 
\ell(Q_i)>0\,,
\] 
until we reach the first index $L$ with $\lambda(Q,Q_L)=0$ or $\ell(Q_L)=0$.  By considering the three possible cases below,  inequality \eqref{eq:lb} may be derived as follows.
\begin{itemize}
\item If $\ell(Q_L)=0$ and $\lambda(Q,Q_L)>0$, then
\[
\sum_{Q_\cM\in\cM}\lambda(Q,Q_\cM)\ge\lambda(Q,Q_L)=2^{\ell(Q)-\ell(Q_L)}  >  1\,,
\]
 since $\ell(Q) > \ell(Q_L)  =0$. 
\item If $\lambda(Q,Q_L)=0$ because $\ell(Q)>\ell(Q_L)+1$, then \eqref{eq: complexity -last} yields $\ell(Q_{L-1})\le\ell(Q_L)+1<\ell(Q)$ and hence
\[\sum_{Q_\cM\in\cM}\lambda(Q,Q_\cM)\ge\lambda(Q,Q_{L-1})=2^{\ell(Q)-\ell(Q_{L-1})}>1.\]
\item If $\lambda(Q,Q_L)=0$ because $\Dist(Q,Q_L)  \ge  2^{1-\ell(Q)}\, C $, then a triangle inequality combined with Lemma~\ref{lma:dist} leads to
\begin{align*}
2^{1-\ell(Q)}\, C  &  \le  
\Dist(Q,Q_1)+\sum_{i=1}^{L-1}\Dist(Q_i,Q_{i+1}) 
\le 2^{-\ell(Q)}\, C +\sum_{i=1}^{L-1} 2^{-\ell(Q_i)}\, C\,.
\end{align*}
 Consequently,  $2^{-\ell(Q)} \le\sum_{i=1}^{L-1} 2^{-\ell(Q_i)}$, and we obtain
\[ 
1\ \le\ \sum_{i=1}^{L-1} 2^{\ell(Q)-\ell(Q_i)}\ =\ \sum_{i=1}^{L-1} \lambda(Q,Q_i)\ \le\ \sum_{Q_\cM\in\cM}\lambda(Q,Q_\cM).\]
\end{itemize}

(ii)\enspace Inequality \eqref{eq:ub} can be derived as follows. For any $0\le j\le J-1$, we consider the set of elements of level $j$ whose distance from $\cQ_{\cM}$ is less than $2^{1-j}\,C$ defined as 
\[
 B(Q_\cM,j)  \,:=\,{\bigl\{Q\in  G^j  : \Dist(Q,Q_\cM)  < 2^{1-j}\,C\}}\,.
\]
 According to the definition of $\lambda$, the set $B(Q_\cM,j)$ collects the elements at level $j$ so that $\lambda(Q,Q_\cM)>0$.  We then have
\begin{align}\label{eq:sumb}
\mathbox[1cm]{\sum_{Q\in\cQ_J\setminus\cQ_0}}\lambda(Q,Q_\cM)
&\le \mathbox[1cm]{\sum_{Q\in\GG\setminus\cQ_0}}\lambda(Q,Q_\cM)
= \sum_{j=1}^{\ell(Q_\cM)+1}2^{j-\ell(Q_\cM)}\,\# 
 B(Q_{\cM},j)\,. 
\end{align}

Since the diagonal of an element $Q$ of level $j$ is $2^{-j}\,\sqrt{d}$, the diagonal of the hypercube composed by the union of the closure of all elements in $B(Q_\cM,j)$ is less or equal to
\[
2\,\cdot\,2^{1-j}\,C\, + 2^{-j}\sqrt{d}=\,
2^{-j}\,\sqrt{d}\,(4\,\tilde{C}+1)\,,
\]
where $\tilde{C}$ is defined in \eqref{eq:c}. Hence, 
\[
\# B(Q_\cM,j) \le (4\,\tilde{C} + 1)^d\,,
\]
and the index substitution $k\coloneqq1-j+\ell(Q_\cM)$ reduces \eqref{eq:sumb} to 
\begin{align*}
\mathbox[1cm]{\sum_{Q\in\cQ_J\setminus\cQ_0}}\lambda(Q,Q_\cM) & \le
\sum_{j=1}^{\ell(Q_\cM)+1}2^{j-\ell(Q_\cM)} \# B(Q_{\cM},j) 
= 
\sum_{k=0}^{\ell(Q_\cM)}2^{1-k} \# B(Q_{\cM},j) \\
& < 2\sum_{k=0}^\infty 2^{-k} \# B(Q_{\cM},j) 
= 4\, \# B(Q_{\cM},j) \le \Lambda\,,
\end{align*}
with $\Lambda=\Lambda(d,p,m)=4(4\tilde{C}+1)^d$. 
\end{proof}
%--------------------------------------------------------------------------------
\section{Conclusions}
\label{sec:conclusions}
 We developed a  complexity estimate which says that the ratio between the refined elements and the marked elements along the refinement history  stays bounded, when refinement is performed as proposed in \cite{buffa2015a}. In particular, this estimate guarantees that the refinement routine ensuring  the (strict) admissibility of the refined mesh remains local at least when looking at the overall refinement process. Note that for a single refinement step, it may be impossible to prove such an estimate \cite{NochettoCIME}. 

Our work paves the way to the analysis of optimal convergence of the adaptive strategy proposed in \cite{buffa2015a} that will be addressed in further studies \cite{buffa2015b}.  
%--------------------------------------------------------------------------------
\section*{Acknowledgements}
Philipp Morgenstern and Daniel Peterseim gratefully acknowledge support by the Deutsche
Forschungsgemeinschaft in the Priority Program 1748 ``Reliable
simulation techniques in solid mechanics. Development of non-standard 
discretization methods, mechanical and mathematical
analysis'' under the project ``Adaptive isogeometric modeling of 
propagating strong discontinuities in heterogeneous materials''. Annalisa Buffa is grateful of the support of the ERC through 
the project HIgeoM - ERC Consolidator Grant n.616563. Carlotta Giannelli acknowledges the support of the Italian MIUR under the FIR project DREAMS (RBFR13FBI3).
%``Design of reliable, exact, and application-oriented technologies for geometric modeling and numerical simulation (DREAMS) ``. 

%------------------------------------------------------------------
%\bibliographystyle{abbrv}
%\bibliography{paper}
\bibliographystyle{plain}
\bibliography{biblio}

\begin{thebibliography}{10}

\bibitem{bazilevs2010}
Y.~Bazilevs, V.~M. Calo, J.~A. Cottrell, J.~Evans, T.~J.~R. Hughes, S.~Lipton,
  M.~A. Scott, and T.~W. Sederberg.
\newblock Isogeometric analysis using {T-Splines}.
\newblock {\em \CMAME}, 199:229--263, 2010.

\bibitem{daveiga2014}
L.~Beir{\~{a}}o~da Veiga, A.~Buffa, G.~Sangalli, and R.~V\'azquez.
\newblock Mathematical analysis of variational isogeometric methods.
\newblock {\em \AN}, 23:157--287, 2014.

\bibitem{binev2004a}
P.~Binev, W.~Dahmen, and R.~DeVore.
\newblock {Adaptive Finite Element Methods with convergence rates}.
\newblock {\em \NM}, 97:219--268, 2004.

\bibitem{bonito2010}
A.~Bonito and R.~H. Nochetto.
\newblock {Quasi-optimal convergence rate of an adaptive discontinuous Galerkin
  method}.
\newblock {\em \SIAMJNA}, 48:734--771, 2010.

\bibitem{bressan2013}
A.~Bressan.
\newblock {Some properties of LR-splines}.
\newblock {\em \CAGD}, 30:778--794, 2013.

\bibitem{buffa2015a}
A.~Buffa and C.~Giannelli.
\newblock Adaptive isogeometric methods with hierarchical splines: error
  estimator and convergence.
\newblock {\em \M3AS}, to appear, 2015.

\bibitem{buffa2015b}
A.~Buffa and C.~Giannelli.
\newblock Adaptive isogeometric methods with hierarchical splines: optimality
  and convergence rates.
\newblock In preparation, 2015.

\bibitem{cfpp2014}
C.~Carstensen, M.~Feischl, M.~Page, and D.~Praetorius.
\newblock Axioms of adaptivity.
\newblock {\em Comput. Math. Appl.}, 67(6):1195--1253, 2014.

\bibitem{cascon2008}
J.~M. Casc\'{o}n, C.~Kreuzer, R.~H. Nochetto, and K.~G. Siebert.
\newblock Quasi-optimal convergence rate for an adaptive finite element method.
\newblock {\em \SIAMJNA}, 46:2524--2550, 2008.

\bibitem{cottrell2009}
J.~A. Cottrell, T.~J.~R. Hughes, and Y.~Bazilevs.
\newblock {\em Isogeometric Analysis: Toward Integration of CAD and FEA}.
\newblock John Wiley \& Sons, 2009.

\bibitem{deng2008}
J.~Deng, F.~Chen, X.~Li, Ch. Hu, W.~Tong, Z.~Yang, and Y.~Feng.
\newblock Polynomial splines over hierarchical {T-meshes}.
\newblock {\em \GM}, 70:76--86, 2008.

\bibitem{dokken2013}
T.~Dokken, T.~Lyche, and K.~F. Pettersen.
\newblock Polynomial splines over locally refined box--partitions.
\newblock {\em \CAGD}, 30:331--356, 2013.

\bibitem{doerfel2010}
M.~R. D\"orfel, B.~J\"uttler, and B.~Simeon.
\newblock Adaptive isogeometric analysis by local h-refinement with
  {T-splines}.
\newblock {\em \CMAME}, 199:264--275, 2010.

\bibitem{evans2015}
E.~J. Evans, M.~A. Scott, X.~Li, and D.~C. Thomas.
\newblock {Hierarchical T-splines: Analysis-suitability, B\'ezier extraction,
  and application as an adaptive basis for isogeometric analysis}.
\newblock {\em \CMAME}, 284:1--20, 2015.

\bibitem{forsey1988}
D.~R. Forsey and R.~H. Bartels.
\newblock Hierarchical {B}-spline refinement.
\newblock {\em \CG}, 22:205--212, 1988.

\bibitem{giannelli2012}
C.~Giannelli, B.~J\"uttler, and H.~Speleers.
\newblock {THB}--splines: the truncated basis for hierarchical splines.
\newblock {\em \CAGD}, 29:485--498, 2012.

\bibitem{giannelli2014}
C.~Giannelli, B.~J\"uttler, and H.~Speleers.
\newblock Strongly stable bases for adaptively refined multilevel spline
  spaces.
\newblock {\em \ACM}, 40:459--490, 2014.

\bibitem{kang2013}
H.~Kang, F.~Chen, and J.~Deng.
\newblock {Modified T-splines}.
\newblock {\em Computer Aided Geometric Design}, 30(9):827--843, 2013.

\bibitem{kraft1997}
R.~Kraft.
\newblock Adaptive and linearly independent multilevel {B}--splines.
\newblock In A.~Le~M{\'e}haut{\'e}, C.~Rabut, and L.~L. Schumaker, editors,
  {\em Surface Fitting and Multiresolution Methods}, pages 209--218. Vanderbilt
  University Press, Nashville, 1997.

\bibitem{kvvb2014}
G.~Kuru, C.~V. Verhoosel, K.~G. van~der Zeeb, and E.~H. van Brummelen.
\newblock Goal-adaptive isogeometric analysis with hierarchical splines.
\newblock {\em \CMAME}, 270:270--292, 2014.

\bibitem{morgenstern2015}
P.~Morgenstern.
\newblock {3D Analysis-suitable T-splines: definition, linear independence and
  $m$-graded local refinement}.
\newblock arXiv:1505.05392, submitted for publication.

\bibitem{mp2015}
P.~Morgenstern and D.~Peterseim.
\newblock {Analysis-suitable adaptive T-mesh refinement with linear
  complexity}.
\newblock {\em \CAGD}, 34:50--66, 2015.

\bibitem{NochettoCIME}
R.~H. Nochetto and A.~Veeser.
\newblock Primer of adaptive finite element methods.
\newblock In {\em Multiscale and adaptivity: modeling, numerics and
  applications}, volume 2040 of {\em Lecture Notes in Math.}, pages 125--225.
  Springer, Heidelberg, 2012.

\bibitem{stevenson2007}
R.~Stevenson.
\newblock Optimality of a standard adaptive finite element method.
\newblock {\em \FCM}, 7:245--269, 2007.

\bibitem{vuong2011}
A.-V. Vuong, C.~Giannelli, B.~J\"uttler, and B.~Simeon.
\newblock A hierarchical approach to adaptive local refinement in isogeometric
  analysis.
\newblock {\em \CMAME}, 200:3554--3567, 2011.

\bibitem{wang2011c}
P.~Wang, J.~Xu, J.~Deng, and F.~Chen.
\newblock {Adaptive isogeometric analysis using rational PHT-splines}.
\newblock {\em \CAD}, 43(11):1438--1448, 2011.

\end{thebibliography}
%------------------------------------------------------------------
\end{document}